\documentclass[11pt]{amsart}
\usepackage[utf8]{inputenc}
\usepackage{amssymb}
\usepackage{amsmath}
\usepackage{mathtools}
\usepackage{ifpdf}
\usepackage{vmargin}

\usepackage{graphicx}  
\usepackage{subfigure}
\usepackage{epstopdf}
\usepackage{caption}

\usepackage{setspace}
\usepackage{xcolor}

\usepackage{vmargin}

\usepackage{amsmath,amssymb,amsthm}
\usepackage{microtype}
\usepackage{xspace}

\usepackage[colorlinks = true,
			citecolor = blue,
			linkcolor = blue,
			urlcolor = black]{hyperref}

\usepackage[]{todonotes}		
\usepackage{mathtools}

\newcommand{\dd}{\mathop{}\!\mathrm{d}}
\let\del\partial

\newcommand{\xaddspace}[0]{\mathchoice{\hspace{-0.8em}}
                {\hspace{-0.6em}}
                {\hspace{-0.4em}}
                {\hspace{-0.3em}}
                 }

\newcommand{\xint}[1]{\int\foreach \i in {2,...,#1}{\xaddspace\int}}
\newcommand{\myiint}[0]{\xint{2}}
\renewcommand{\iint}{\myiint}

\let\Re\relax

\DeclareMathOperator{\Re}{Re}

\newcommand{\ee}{\textup{e}}
\newcommand{\ip}[1]{\langle#1\rangle}
\usepackage{scalerel,stackengine}
\stackMath
\newcommand\widecheck[1]{%
\savestack{\tmpbox}{\stretchto{%
  \scaleto{%
    \scalerel*[\widthof{\ensuremath{#1}}]{\kern-.6pt\bigwedge\kern-.6pt}%
    {\rule[-\textheight/2]{1ex}{\textheight}}
  }{\textheight}%
}{0.5ex}}%
\stackon[1pt]{#1}{\scalebox{-1}{\tmpbox}}%
}

\DeclareSymbolFont{bbold}{U}{bbold}{m}{n}
\DeclareSymbolFontAlphabet{\mathbbold}{bbold}

\newcommand{\coloneq}{\mathrel{\mathop:}=}

\newcommand{\AH}{H_{\textup{par}}}

\newtheorem{thm}{Theorem}[section]
\newtheorem{cor}[thm]{Corollary}
\newtheorem{lem}[thm]{Lemma}
\newtheorem{prop}[thm]{Proposition}
\theoremstyle{definition}
\newtheorem{defn}[thm]{Definition}

\theoremstyle{remark}
\newtheorem{rem}[thm]{Remark}
\numberwithin{equation}{section}

\DeclareMathOperator{\sech}{sech}
\newcommand\homog\mathring

\begin{document}
	\onehalfspacing

\title[Riesz transforms and Sobolev spaces associated to $\AH$]{Riesz transforms and Sobolev spaces associated to the partial harmonic oscillator}


\author[X. Su]{Xiaoyan Su}
\address{Loughborough University,\ Epinal Way, \ Loughborough, \ Leicestershire,\ UK, \ LE11 3TU.}
\email{x.su2@lboro.ac.uk}


\author[Y. Wang]{Ying Wang}
\address{Mazarredo Zumarkalea, 14, Abando, \ Bilbao, \ Bizkaia, \ Spain,   \ 48009.}
\email{ywang@bcamath.org}

\author[G. Xu]{Guixiang Xu}
\address{Laboratory of Mathematics and Complex Systems,\
Ministry of Education,\
School of Mathematical Sciences,\
Beijing Normal University,\
Beijing, 100875, People's Republic of China.
}
\email{guixiang@bnu.edu.cn}



\begin{abstract}

In this paper, our goal is to establish the Sobolev space associated to the partial harmonic oscillator. Based on its heat kernel estimate, we firstly give the definition of the fractional powers of the partial harmonic oscillator 
$$H_{\textup{par}}=-\partial_{\rho}^2-\Delta_x+|x|^2,$$
 and show that its negative powers are well defined on $L^p(\mathbb R^{d+1})$ for $p\in [1,\infty]$.   We then define associated Riesz transforms and show that they are bounded on classical Sobolev spaces by the calculus of symbols.

Secondly, by a factorization of the operator $H_{\textup{par}}$, we define two families of Sobolev spaces with positive integer indices, and show the equivalence between them by the boundedness of Riesz transforms. Moreover, 
the adapted symbolic calculus also implies the boundedness of Riesz transforms on the Sobolev spaces associated to the partial harmonic oscillator $H_{\textup{par}}$.

Lastly, as applications of our results, we obtain the revised  Hardy--Littlewood--Sobolev inequality, the Gagliardo--Nirenberg--Sobolev inequality, and Hardy's inequality in the potential space $L_{H_{\textup{par}}}^{\alpha, p}$.
\end{abstract}
 \maketitle

\section{Introduction} In this paper, we consider the following Schr\"odinger operator in $\mathbb R^{d+1}$,\begin{equation*}
     \AH=-\partial_{\rho}^2-\partial_{x_1}^2-\dots-\partial_{x_d}^2+|x|^2,
\end{equation*}
which we will call a partial harmonic oscillator. The anisotropy of the operator is used in \cite{AHS, Josser} to model magnetic traps in the Bose--Einstein condensation.
  It is well known in \cite{AHS} that $\AH$ is a self-adjoint operator,  $\sigma ( \AH)=\sigma_{\textup{ac}}( \AH)=[d, \infty) $  with embedded generalized eigenvalues $2k+d$ for $k\geq 1$, and that there is no singular spectrum.

In this paper, we investigate the Riesz transforms and  Sobolev spaces adapted to the operator $\AH$. The  fractional integrals and the Riesz transform associated to the classical Laplacian  $-\Delta_{\mathbb R^{d+1}}$ are fundamental in harmonic analysis and have significant applications in the study of partial differential equations, see \cite{Gra250, MSchlag, Stein} for example. 

Our work is motivated by the series of papers \cite{BT06, KMVZZ, KMVZZ2017,  KVZ09, KVZ, LiSZ, LiXZ, MMZ, Thangavelup87,  Thangavelup90, Thangavelu18}.  Especially in \cite{KMVZZ}, R. Killip, etc,  studied the Sobolev spaces adapted to the Schr\"odinger equation with the inverse-square potential by using the heat kernel estimates. These results were crucially used in \cite{KMVZZ2017, MMZ} to obtain the scattering result of the solution for nonlinear Schr\"odinger and wave equations with the inverse-square potential. Riesz transforms outside a convex  obstacle were studied in \cite{KVZ}. Other works can be found in \cite{KVZ09, LiSZ, LiXZ}.  The  Riesz transforms for the Hermite operators were also investigated in \cite{Thangavelup90,Thangavelu18}, some techniques in this paper are also inspired by that in \cite{BT06} where the Sobolev spaces adapted to the Hermite operator are discussed. 

Firstly, based on the Fourier--Hermite expansion and the heat kernel estimate for the operator $\AH$,  the fractional powers  $\AH^\alpha$ can be well-defined on $C_0^\infty(\mathbb R^{d+1})$ for $\alpha\in\mathbb R$.
For the operator $\AH^{\alpha}$ with the negative powers  ($\alpha<0$), we can obtain that its integral kernel  $K_\alpha(z,z')$ is controlled by an integrable function of $z-z'$, and therefore $\AH^{\alpha}$ can be  extended to an operator on $L^p(\mathbb R^{d+1})$ for $p\in [1, \infty]$.


After the study of fractional powers,  we can consider the potential spaces  $L_{\AH}^{\alpha, p}=\AH^{-\alpha/2}(L^p(\mathbb R^{d+1}))$, and  define the Sobolev spaces $W_{\AH}^{k,p}$, $k\in\mathbb N$ by using a factorization of the operator $\AH$. It turns out that they are  equivalent for $\alpha=k\in\mathbb N$, due to the $L^p$-boundedness of Riesz transforms  in the classical Sobolev spaces. Moreover, we can also show that the Riesz transforms are bounded on $L_{\AH}^{\alpha, p}$ by the symbolic calculus in class $G^m$. We will also compare three Sobolev spaces associated to $-\Delta_{\mathbb R^{d+1}}$, the Hermite operator $H=-\Delta_{\mathbb R^{d+1}}+\rho^2+|x| ^2$, and the operator $\AH$ respectively. Finally, we use the above results to show the revised Hardy--Littlewood--Sobolev inequality,  Gagliardo--Nirenberg--Sobolev inequality and Hardy's inequality in  the potential space $L_{\AH}^{\alpha, p}$.

 We remark that the operator $\AH$ is a polynomial perturbation of the Laplacian and some results have been established in \cite{Dziubanski09} by using the singular integral operators on nilpotent Lie groups.  Instead, our results  closely rely on the heat kernel estimate for the operator $\AH$ and  Mehler's formula.
By the way, the authors will establish a Mikhlin--H\"ormander multiplier theorem for the operator $\AH$ and obtain its applications in the Littlewood--Paley square function estimate together with the Khintchine inequality in the subsequent paper \cite{SuWangXu}.

The results in this paper can be adapted to obtain analogous results for  general partial harmonic oscillators 
$ L=-\Delta_x-\Delta_y+|x|^2, \ (x,y) \in \mathbb R^{d_1} \times \mathbb R^{d_2}.$ There are some scattering results for nonlinear Schr\"odinger equations with partial harmonic potentials, see \cite{ACS, Chengguo}, and some regularity problem of the fundamental solutions for the Schr\"odinger flows with partial harmonic oscillator has been discussed by Zelditch in \cite{Zelditch}.

Lastly, this paper is organized as follows: some preliminary results and adapted symbol class  are introduced in Section \ref{s:2-prelim}.  In Section \ref{s:3-frac-powers}, by the symblolic calculus, we discuss the fractional powers and the Riesz  transforms associated to the operator $\AH$.  In Section \ref{s:4-sobolev-spaces}, we define the Sobolev spaces and discuss some inclusion properties. In Section \ref{s:5-fun-ineqs}, we  show the revised Hardy--Littlewood--Sobolev inequality,  Gagliardo--Nirenberg--Sobolev inequality, and Hardy's inequality in the space $L_{\AH}^{\alpha, p}$.

\subsection*{Acknowledgements}
G. Xu is supported by NSFC (No. 112371240, No. 12431008) and  by National Key Research and Development Program of China (No. 2020YFA0712900).

\section{Preliminaries}
\label{s:2-prelim}
\subsection{Notation}
Let $D\coloneq\partial/i$. We write $z=(\rho, x)$ or $z'=(\rho', x)$ to denote elements in $\mathbb R^{d+1}$ with $\rho, \rho' \in \mathbb R$ and $x, x' \in \mathbb R^d$. We write $A\lesssim B$ if $A \leq C B$ for some constant $C>0$ and $A \gtrsim B$ if $A \geq c B$ for some constant $c>0$. We write the Japanese bracket as $\ip{x}=(1+|x|^2)^{1/2}$. 

We use $H$ and $\Delta_{\mathbb R^{d+1}}$ to denote the Hermite and Laplacian operators in $\mathbb R^{d+1}$ respectively, and  $\Delta$ to denote the Laplacian operator in $\mathbb R^d$, i.e.
\begin{align*}
    H=-\partial_\rho^2-\Delta +\rho^2+|x|^2, \quad -\Delta_{\mathbb R^{d+1}}=-\partial_\rho^2-\Delta.
\end{align*}
For any $\alpha>0$, we define the classical Sobolev space $W^{\alpha,p}(\mathbb R^{d+1})$ and the Hermite--Sobolev space $L_H^{\alpha,p}(\mathbb R^{d+1})$ as those  in \cite{BT06, Gra250, MSchlag, Stein}  as follows:
\begin{align*}
    W^{\alpha, p}(\mathbb R^{d+1})&=\{f\in L^p(\mathbb R^{d+1}): (1-\Delta_{\mathbb R^{d+1}})^{-\alpha/2} f\in L^p(\mathbb R^{d+1})\},\\
    L_{H}^{\alpha, p}(\mathbb R^{d+1})&=\{f\in L^p(\mathbb R^{d+1}):H^{-\alpha/2} f\in L^p(\mathbb R^{d+1})\}.
\end{align*}

The Fourier transform and the inverse Fourier transform of a Schwartz function $f\in \mathcal S(\mathbb R)$ are defined by 
\begin{align*}
    \mathcal F^{\pm}_{\rho}f (\tau) =\frac{1}{(2\pi)^{1/2}}\int_{\mathbb R} e^{\mp i  \rho \cdot \tau}  f(\rho) \dd \rho, 
\end{align*}
and they are similar in the higher dimensional cases.

\subsection{Hermite functions} We recall some basic results about Hermite functions from \cite{BGV, folland, Thangavelu}. For $k=0, 1, \ldots,$ the Hermite functions $h_k(x)$ on $\mathbb R$ are defined by
 \begin{align*}
     h_k(x)=(2^k k! \sqrt{\pi})^{-1/2} (-1)^k \frac{\dd^k}{\dd x^k}(\ee^{-x^2}) \ee^{-x^2/2}, 
 \end{align*}
which form a complete orthonormal basis for $L^2(\mathbb R)$. For $\mu \in \mathbb N^d$, the Hermite functions $\Phi_\mu(x)$  on $\mathbb R^d$ are defined by taking the product of the $1$-dimensional Hermite functions $h_{\mu_j}$, 
\begin{align*}
    \Phi_\mu(x)=\prod_{j=1}^d h_{\mu_j}(x_j).
\end{align*}
It is well known  that they are eigenfunctions of the operator $-\Delta+|x|^2$ with eigenvalues $2|
 \mu|+d$,
 \begin{align*}
   (- \Delta+|x|^2 )\Phi_\mu(x)= (2|\mu|+d)\Phi_\mu(x).
 \end{align*}
 Let us define the following first order differential operators for $1\leq j \leq d$ as
\begin{align*}
A_0=-\partial_{\rho}, \quad A^*_0=\partial_{\rho}, \ A_j=-\frac{\partial}{\partial x_j}+x_j, \ A_{-j}= A_j^*=\frac{\partial}{\partial {x_j}}+x_j.
    \end{align*}
Then we have
\begin{align}\label{derivative of Hermite function}
A_j \Phi_{\mu}  =\sqrt{2(\mu_j+1)}  \Phi_{\mu+e_j},  \;\text{and}\;   A_{-j} \Phi_{\mu}  =\sqrt{2\mu_j}  \Phi_{\mu-e_j},
\end{align}
where $e_j$ is the $j$th unit vector in $\mathbb N^d$.
We use $A_j$ and its adjoint operator $A^{*}_j$ to  factorize the operator $ \AH$ as follows,
\begin{align*}
  \AH=\frac{1}{2} \sum_{j=0} ^d (A_j A_j^*+A_j^* A_j).
\end{align*}
Denote by $P_k$ the spectral projection to the $k$th eigenspace of $-\Delta+|x|^2$,
 \begin{align}
P_kf(x)=\int_{\mathbb R^d} \sum_{|\mu|=k}\Phi_\mu(x)\Phi_\mu(x') f(x')\dd x'. \label{e:P_k}
 \end{align}
 These projections  are integral operators with kernels
 \begin{align*}
     \Phi_k(x,x')=\sum_{|\mu|=k}\Phi_\mu(x)\Phi_\mu(x').
 \end{align*}
The useful Mehler formula for $ \Phi_k(x,x')$ is
\begin{align}\label{Mehler's formula}
   \sum_{k=0} ^\infty r^k \Phi_k(x,x') =\pi ^{-d/2}(1-r^2)^{-d/2} e^{-\frac{1}{2}\frac{1+r^2}{1-r^2}(|x|^2+|x'|^2)+\frac{2r x \cdot x'}{1-r^2}},
\end{align}
for $0<r<1$, please refer to \cite{BGV, Thangavelu}.
\subsection{Symbol class}
First, we recall the definition of the standard symbol class. We denote the  spatial variable by $z=(\rho,x)$, and the frequency variable by $\omega = (\tau,\xi)$.

\begin{defn}[{\cite{Taylor, Taylor91}}] Let $0\leq \epsilon, \delta \leq 1$. We say a function $\sigma(z,\omega) \in C^{\infty}(\mathbb R^{d+1} \times \mathbb R^{d+1})$ is a symbol of order $m$ with type $(\epsilon, \delta)$, written $\sigma \in S^m_{\epsilon,\delta}$ if for all multi-indices $\beta, \gamma$, there is a constant $C_{\beta, \gamma}$ such that
\begin{align*}
  |D_{z}^{\beta} D_{\omega}^\gamma \sigma(z,\omega)|   \leq C_{\beta, \gamma}\ip{ \omega}^{m-\epsilon |\gamma|+\delta |\beta|}.
\end{align*}

\end{defn}

\begin{defn}[{\cite{NicolaRodino}}] We say a function $\sigma(z, \omega) \in C^{\infty}(\mathbb R^{d+1} \times \mathbb R^{d+1})$ belongs to $\Gamma^m$ if for all multi-indices $\beta, \gamma$, there is constant $C_{\beta, \gamma}$   such that  \begin{align*}
    |D_{z}^{\beta} D_{\omega}^\gamma \sigma( z,\omega)|   \leq C_{\beta, \gamma}\ip{|z|+|\omega|}^{m- |\beta|-|\gamma|}.
\end{align*}
\end{defn}

\begin{prop}[\cite{Thangavelu18}] Let $\alpha<0$.  The symbols of the negative fractional powers $H^{\alpha}$ of the Hermite operators  belong to $\Gamma^{2\alpha}$ .
\end{prop}

For later use, we need introduce a new symbol class, adapted to the operator $\AH$, in which the upper bounds also depend on partial spatial variable $x$.

\begin{defn}\label{generalized symbol} We say that a function $\sigma(\rho, x, \tau, \xi) \in C^{\infty}(\mathbb R^{d+1} \times \mathbb R^{d+1})$ belongs to $G^m$ if for all multi-indices $\beta, \gamma$, there is constant $C_{\beta, \gamma}$   such that  \begin{align*}
    |D_{z}^{\beta} D_{\omega}^\gamma \sigma(\rho, x, \tau, \xi)|   \leq C_{\beta, \gamma}\ip{|x|+|\omega|}^{m- |\beta|-|\gamma|}.
\end{align*}
\end{defn}
It is easy to see that   $\Gamma^m \subset G^m \subset S_{1,0}^{m}$ for $m\leq 0$.

We  quantize a symbol in  $ \sigma \in G^m$ by defining
\begin{align*}
    T_{\sigma} f(z)=\frac{1}{(2\pi)^{d+1}} \iint_{{\mathbb R^{d+1}\times \mathbb R^{d+1}}}\ee^{i (z-z') \omega} \sigma (z, \omega) f(z') \dd z' \dd \omega.
\end{align*}
Using integration by parts, we know that $T_{\sigma}$ is bounded from $\mathcal S(\mathbb R ^{d+1})$ to  $\mathcal S(\mathbb R ^{d+1})$ if  $ \sigma \in G^m$.
For $p\in G^{m_1}, q \in G^{m_2}$,  define $ p \# q(z,\omega)$  by the oscillatory integral
    \begin{align*}
        p \# q(z, \omega) \coloneq &\frac{1}{(2\pi)^{d+1}} \iint_{\mathbb R^{d+1} \times \mathbb R^{d+1}} \ee^{-i z' \omega'}  p(z, \omega+\omega') q(z+z', \omega) \dd z' \dd \omega'.
         \end{align*}

\begin{prop} \label{composition}
If $p\in G^{m }$, $q \in G^{m'} $ then
  $ p \# q\in G^{m+m'}$. In particular, if $m+m'\le0$, then $p\#q \in S^0_{1,0}$.
\end{prop}
\begin{proof}
 It suffices to prove that $|p\#q| \lesssim \ip{|x|+|\omega|}^{m+m'}$. In fact,  the derivative estimates of $p\#q$ will follow from the fact that $D^\alpha p \in G^{ m-|\alpha|}$ and   $D^\beta q \in G^{m'-|\beta|}$.
    By definition, we have for any sufficiently large $k$ that,
\begin{align*}
     &p\# q(z,\omega) = \iint_{ \mathbb R^{d+1}\times\mathbb R^{d+1}} F \dd z' \dd \omega'
     \\
     &= \frac{1}{(2\pi)^{d+1}} \iint_{\mathbb R^{d+1}\times\mathbb R^{d+1}} \!\!\ee^{-i z' \omega'}  \ip{\nabla_{z'}}^{2k} \frac{ q(z+z', \omega)}{\ip{ z'}^{2k} }
 \ip{\nabla_{\omega'}}^{2k}  \frac{p(z, \omega+\omega')}{ \ip{\omega'} ^{2k}} \dd z' \dd \omega'.
\end{align*}
The integrand $F$ is controlled by \[ \frac{ p(z,\omega+\omega') q(z+z', \omega)}{\ip{ z'}^{2k} \ip{ \omega'}^{2k} } \lesssim \frac{\ip{|x| + |\omega+\omega'| }^m \ip{|x+x'|+|\omega|}^{m'}  }{\ip{ z'}^{2k} \ip{ \omega'}^{2k} } . \]
\subsubsection*{Case 1: $m\ge 0,m'\ge 0$}  Peetre's inequality immediately gives for $k\gg1$
\[ |p\# q| \lesssim  \ip{|x|+|\omega|}^{m+m'} \iint_{ \mathbb R^{d+1}\times\mathbb R^{d+1}}\! \ip{z'}^{m'-2k}\ip{\omega'}^{m-2k}\dd z' \dd \omega' \lesssim \ip{|x|+|\omega|}^{m+m'} . \]
\subsubsection*{Case 2: $m\ge 0,m'<0$}
We partition
$ \mathbb R^{d+1}_{z'} = A_< \cup A_\ge,$ where

\begin{align*}
    A_{<}  = \left \{ |z'|<\frac{|x|+|\omega|}2\right\} ,\qquad  A_{\ge} =  \left \{ |z'|\ge\frac{|x|+|\omega|}2\right\},
\end{align*}

\textbf{Estimate On $A_<$}: note that $|x+x'|+|\omega|\ge |x|-|x'|+|\omega| \ge \frac{|x|+|\omega|}2$, we have $ \ip{|x+x'|+|\omega|}^{m'} \lesssim \ip{ |x| + |\omega|}^{m'}  $ and

\[ \left|\iint_{A_<\times \mathbb R^{d+1}} F \dd z' \dd \omega'\right| \lesssim  \ip{ |x| + |\omega|}^{m+m'} \iint_{\mathbb R^{d+1}\times\mathbb R^{d+1}} \ip{ z'}^{-2k}\ip{ \omega'}^{-2k+m} \dd z' \dd\omega'.  \]

\textbf{Estimate On $A_\ge$}: we have  $\ip{|x+x'|+|\omega|}^{m'}\le 1$ and  use the denominator to control the $\ip{|x| + |\omega| }^m$ term:
\[ \left|\iint_{A_\ge\times\mathbb R^{d+1} } F \dd z' \dd \omega' \right|\lesssim \ip{ |x| + |\omega|}^{m+m'}\iint_{\mathbb R^{d+1}\times\mathbb R^{d+1}} \ip{ z'}^{-2k-m'}\ip{ \omega'}^{-2k} \dd z' \dd \omega'.\]
 
Both bounds are controlled by $ \ip{|x|+ |\omega| }^{m+m'}$ if $k$ is large enough.

\subsubsection*{Case 3: $m<0$, $m'\ge 0$}
We decompose $\mathbb R_{\omega'}^{d+1} = B_< \cup B_\ge $, where \[ B_{<}  = \left \{ |\omega'|<\frac{|x|+|\omega|}2\right\} , \qquad B_{\ge} =  \left \{ |\omega'|\ge\frac{|x|+|\omega|}2\right\}. \]
On the set $B_<$,  we have $\ip{ |x|+|\omega+\omega'|}^{m} \lesssim \ip{ |x|+|\omega|}^{m}$, and
on the set $B_\geq $ we have
$
    \ip{ \omega'}^m \lesssim \ip{ |x|+|\omega|}^{m},
$ so we can obtain the result by  similar argument in \emph{Case 2} (despite the fact that the required upper bounds depend on $\tau$, but not on $\rho$).

\subsubsection*{Case 4: $m<0$, $m'<0$} We use the decompositions as follows $$ \mathbb R^{d+1}\times \mathbb R^{d+1} =(A_<\times  B_< ) \cup (  A_{<}\times B_\ge) \cup (  A_\ge \times B_<) \cup ( A_\ge \times  B_\ge). $$
For the above decomposition, we have the following estimates:
\begin{align*}
   A_<: &\ \ip{|x+x'|+|\omega|}^{m'} \lesssim \ip{ |x| + |\omega|}^{m'},\\ 
   A_>: &\  \ip{|x+x'|+|\omega|}^{m'} \lesssim 1,\\
   B_<: & \ \ip{ |x|+|\omega+\omega'|}^{m} \lesssim \ip{ |x|+|\omega|}^{m},\\
   B_>: &\  \ip{ |x|+|\omega+\omega'|}^{m} \lesssim 1.\end{align*}
By combining the above estimates on the respective sets, we have
\begin{gather*}
\left| \iint_{A_< \times B_< } F \dd z' \dd \omega' \right|\lesssim \ip{ |x| + |\omega|}^{m+m'}\iint_{\mathbb R^{d+1}\times\mathbb R^{d+1}} \ip{ z'}^{-2k}\ip{ \omega'}^{-2k} \dd z' \dd \omega',\\
          \left|\iint_{A_\ge \times B_\ge  } \!\!F \dd z' \dd \omega'\right| \lesssim \ip{ |x| + |\omega|}^{m+m'}\iint_{\mathbb R^{d+1}\times\mathbb R^{d+1}}\!\! \ip{ z'}^{-2k-m'}\ip{ \omega'}^{-2k-m} \dd z' \dd \omega'.\end{gather*}
The integrals over $A_<\times B_\ge$ and $A_\ge \times B_<$ are similar, and both of them are bounded by $\ip{ |x| + |\omega|}^{m+m'}$ if $k$ is sufficiently large.

 This completes the proof.
 \end{proof}

\section{Fractional powers of  the operator $ \AH$}
\label{s:3-frac-powers}
\subsection{Functional Calculus for the operator $ \AH$ } Fix $z=(\rho, x) \in \mathbb R^{d+1}$. By the continuous Fourier transform in $\rho \in \mathbb R$  and the discrete Hermite expansion in $x \in \mathbb R^{d}$ of the operator $\AH$,  we  can write $\AH f$ for a function $f\in C_0^\infty(\mathbb R^{d+1})$ as
\begin{align*}
     \AH f(\rho,x)&=
    \sum_{\mu} \frac{1}{\sqrt{2\pi}}\int_{\mathbb R} \ee^{i \rho \tau}(\tau^2+2|\mu|+d)  (\mathcal F_{\rho} f(\tau, \cdot), \Phi_{\mu}(\cdot)) \Phi_{\mu} (x)\dd \tau\\
   &= \sum_{k=0}^\infty\frac{1}{\sqrt{2\pi}} \int_{\mathbb R} \ee^{i \rho\tau }(\tau^2+2k+d)P_k \mathcal F_{\rho} f(\tau, x) \dd \tau,
    \end{align*}
where $\mathcal F_{\rho} f$ is the Fourier transform with respect to $\rho$,  and $P_k$ is the projection operator in \eqref{e:P_k}.
Thus, for a Borel measurable function $F$ defined on $\mathbb R_{+}$, we can define the operator $F(\AH)$ by the spectral theory as
\begin{align}\label{FCofAH}
    F(\AH )f(\rho, x)=\sum_{k=0}^\infty \frac{1}{\sqrt{2\pi}} \int_{\mathbb R} \ee^{i \rho \tau}F(\tau^2+2k+d)P_k \mathcal F_{\rho} f(\tau, x) \dd \tau,
\end{align}
so long as the right hand side makes sense.

In particular,  the heat semigroup $\ee^{-t{\AH}}$ can be  defined by
\begin{align}
    \ee^{-t \AH}f(\rho,x)&=  \sum_{\mu}\frac{1}{\sqrt{2\pi}} \int_{\mathbb R}\ee^{i\tau \rho} \ee^{-t(\tau^2+2|\mu|+d)  } ((\mathcal F_{\rho} f)(\tau, \cdot ), \Phi_{\mu}(\cdot)) \Phi_{\mu} (x)\dd \tau\notag \\
    &=\sum_{k=0}^\infty \frac{1}{\sqrt{2\pi}}\int_{\mathbb R} \ee^{i\tau \rho} \ee^{-t(\tau^2+2k+d)  }  {P_k}(\mathcal F_{\rho} f)(\tau, x)  \dd \tau\label{heat-formula-1}
\end{align}
for any $f\in C_0^\infty(\mathbb R^{d+1})$.
By Mehler's formula \eqref{Mehler's formula}, the integral kernel of the operator $\ee^{-t \AH}$ is
\begin{align}\label{kernel of semigroup}
   E(t,z,z')
    =2^{-\frac{d+2}{2}}\pi^{-\frac{d+1}{2}}t^{-1/2}(\sinh 2t)^{-d/2} \ee^{-B(t,z,z')},
\end{align}
where
\begin{align*}
    B(t, z, z')=\frac{1}{4}(2\coth2t-\tanh t)|x-x'|^2+\frac{\tanh t }{4}|x+x'|^2+\frac{(\rho-\rho')^2}{4t}.
\end{align*}

\subsection{Fractional powers of the operator $\AH$ and the heat semigroup}\label{sub: heat kernel rep}On the one hand,
for  $\alpha \in \mathbb R$, we can define the fractional powers $\AH^\alpha$ on $C_0^\infty(\mathbb R^{d+1})$ by
\begin{align*}
\AH^{\alpha}f(\rho,x)&=\sum_{k=0}^\infty\frac{1}{\sqrt{2\pi}} \int_{\mathbb R} \ee^{i\tau \rho}(\tau^2+2k+d)^{\alpha}  {P_k} (\mathcal F_{\rho} f)(\tau, x)  \dd \tau.
\end{align*}
Simple calculation shows that the identity 
\begin{align}\label{power property}
\AH^{\alpha} \cdot    \AH^{\beta}f = \AH^{\alpha+\beta} f,
\end{align}
holds for  $ f\in C_0^\infty(\mathbb R^{d+1})$ and  $\alpha,\beta\in\mathbb R$. 

On the other hand,  we can also formulate  the powers $ \AH^{ \alpha}$  with $\alpha\in \mathbb R$  by  the semigroup $\ee^{-t \AH}$ and the Gamma function.  Firstly, for any $\alpha>0$, the negative powers $ \AH^{-\alpha}$ can be written as 
 \begin{align}\label{negative power and semigroup}
      \AH^{-\alpha} f(\rho, x) =\frac{1}{ \Gamma(\alpha)} \int_0^\infty t^{\alpha-1} \ee^{-t \AH} f(\rho, x)  \dd t.
 \end{align}
Notice that the integral kernel of the operator $\AH^{-\alpha}$ is positive  since the integral kernel  \eqref{kernel of semigroup} of $\ee^{-t \AH} $ is positive. Similarly for any $a\in \mathbb R$ and $d>- a$, we have
 \begin{align}\label{kernel of H+a}
     (\AH+a)^{-\alpha} f(\rho, x) =\frac{1}{ \Gamma(\alpha)} \int_0^\infty t^{\alpha-1}\ee^{-t a} \ee^{-t \AH} f(\rho, x)  \dd t
 \end{align}
so long as the integral exists.

We now express the positive fractional powers in terms of derivatives of the semigroup. Let $N$ be the smallest integer which is larger than $\alpha$. Using the identity  $-\AH \ee^{-t\AH}=\frac{\dd}{\dd t} \ee^{-t\AH}$ and \eqref{power property}, we have for any $f\in C_0^\infty(\mathbb R^{d+1})$ that 
\begin{align}\label{positive power and semigroup}
    \AH^{\alpha} f(\rho, x)=\frac{(-1)^N}{\Gamma(N-\alpha)}\int_{0}^\infty t^{N-\alpha-1} \frac{\dd^N }{\dd t^N} \ee^{-t\AH}f(\rho, x) \dd t.
\end{align}

We have the following properties between $\AH^{\alpha}$ and $A_{\pm j}$ for $1\leq j\leq d$.
\begin{lem} \label{commute lemma}
 For any $\alpha\in \mathbb R$, $d\ge3$, $1\le j\le d$, and $f\in C_0^\infty(\mathbb R^{d+1})$, we have
\begin{align*}
    A_0  \AH^{\alpha} f &=   \AH^{\alpha} A_0f,\\
    A_j  \AH^{\alpha}f &=( \AH-2)^{\alpha} A_j f, &  A_{-j}  \AH^{\alpha} f =( \AH+2)^\alpha A_{-j} f,\\
       \AH^{\alpha} A_j f&= A_j ( \AH+2)^\alpha f,
       &   \AH^{\alpha}   A_{-j} f =A_{-j} ( \AH-2)^{\alpha}f.
        \end{align*}
\end{lem}

\begin{proof} The first result  is trivial. For $1\leq j\leq d$, we only give the details for $ A_j \AH^{\alpha}f =( \AH-2)^{\alpha} A_j f$, as the other cases are dealt with in analoge argument.
 By the definition of $\AH^{\alpha}$, we have 
\begin{align*}
    \AH^{\alpha} f(\rho, x) =\sum_{\mu}\frac{1}{\sqrt{2\pi}} \int_{\mathbb R} \ee^{i\tau \rho}(\tau^2+2|\mu|+d)^\alpha (\mathcal F_{\rho} f(\tau, \cdot), \Phi_{\mu}(\cdot)) \Phi_{\mu} (x)\dd \tau.
\end{align*}
On one hand,  by  \eqref{derivative of Hermite function}, we have
\begin{align*}
    &A_j  \AH^{\alpha} f(\rho, x)\\
    &=\sum_{\mu} \frac{1}{\sqrt{2\pi}}\int_{\mathbb R} \ee^{i\tau \rho}(\tau^2+2|\mu|+d)^\alpha  (\mathcal F_{\rho} f(\tau, \cdot), \Phi_{\mu}(\cdot ))A_j \Phi_{\mu} (x)\dd \tau\\
    &=\sum_{\mu}\frac{1}{\sqrt{2\pi}} \int_{\mathbb R} \ee^{i\tau \rho}(\tau^2+2|\mu|+d)^\alpha  (\mathcal F_{\rho} f(\tau, \cdot), \Phi_{\mu}(\cdot))\dd \tau\cdot  \sqrt{2(\mu_j+1)}\Phi_{\mu+e_j}(x).
\end{align*}
On the other hand, by the fact that   $A_{-j}=A_j^*$, we obtain
\begin{align*}
A_jf(\rho, x)&=\sum_{\mu}\frac{1}{\sqrt{2\pi}} \int_{\mathbb R} \ee^{i\tau \rho}(A_j\mathcal F_{\rho} f(\tau, \cdot), \Phi_{\mu}(\cdot)) \Phi_{\mu} (x)\dd \tau\\
&=\sum_{\mu}\frac{1}{\sqrt{2\pi}} \int_{\mathbb R} \ee^{i\tau \rho}(\mathcal F_{\rho} f(\tau, \cdot), A_{-j}\Phi_{\mu}(\cdot)) \Phi_{\mu} (x)\dd \tau\\
&=\sum_{\mu}\frac{1}{\sqrt{2\pi}} \int_{\mathbb R} \ee^{i\tau \rho}(\mathcal F_{\rho} f(\tau, \cdot),  \Phi_{\mu-e_j}(\cdot))\dd \tau \cdot \sqrt{2\mu_j} \Phi_{\mu} (x).
\end{align*}
By \eqref{FCofAH},  we have for a function $g\in C_0^\infty(\mathbb R^{d+1})$ that
\begin{align*}
   ( \AH-2)^{\alpha} g=\sum_{\mu} \frac{1}{\sqrt{2\pi}}\int_{\mathbb R} \ee^{i\tau \rho}(\tau^2+2|\mu|+d-2)^\alpha  (\mathcal F_\rho g(\tau, \cdot), \Phi_{\mu}(\cdot)) \Phi_{\mu} (x)\dd \tau, 
\end{align*}
which implies by choosing $g=A_jf$ that
\begin{align*}
 & ( \AH-2)^{\alpha}  A_jf(\rho, x)\\
 &=\sum_{\mu} \frac{1}{\sqrt{2\pi}}\int_{\mathbb R} \ee^{i\tau \rho}(\tau^2+2|\mu|+d-2)^\alpha  (\mathcal F_{\rho} f(\tau, \cdot),  \Phi_{\mu-e_j}(\cdot))\dd \tau\sqrt{2\mu_j} \Phi_{\mu} (x)\\
 &=\sum_{\mu}\frac{1}{\sqrt{2\pi}} \int_{\mathbb R} \ee^{i\tau \rho}(\tau^2+2|\mu|+d)^\alpha (\mathcal F_{\rho} f(\tau, \cdot),  \Phi_{\mu}(\cdot ))\dd \tau\sqrt{2\mu_j+2} \Phi_{\mu+e_j} (x)\\
 &= A_j  \AH^{\alpha}f(\rho, x).
 \end{align*}
 
We finish the proof.
\end{proof}

\subsection{Properties of negative fractional powers of $\AH$}
In this subsection, we explore the \ properties of the negative powers of the operator $\AH$. The first result is:
\begin{prop} \label{prop31} Given $\alpha>0$,   the operator $ \AH^{-\alpha}$ has the integral representation
\begin{align*}
     \AH^{-\alpha} f(z)=\int_{\mathbb R^{d+1}}{K_{\alpha}}(z,z') f(z')\dd z' \end{align*}
for all $f \in C_0^\infty(\mathbb R^{d+1})$. Moreover, there exist a functions $\Psi_\alpha \in L^1(\mathbb R^{d+1})$ and a constant $C>0$ such that
\begin{align}
 {K_{\alpha}}(z, z')  \leq C  \Psi_\alpha(z-z'), \ \text{for all}\  z, z' \in \mathbb R^{d+1}. \label{kernel-ineq}
\end{align}
Hence, $\AH^{-\alpha}$ is well defined  and bounded on $L^p(\mathbb R^{d+1})$ for $p \in [1, +\infty]$.
\end{prop}

\begin{proof} By \eqref{kernel of semigroup} and \eqref{negative power and semigroup}, we have
\begin{align*} K_{\alpha}(\rho, x, \rho' ,x') =  \frac{C_d}{\Gamma(\alpha)} \int_0^\infty t^{\alpha-1-1/2} (\sinh 2t)^{-d/2} \ee^{-B(t, z, z')}  \dd t.
\end{align*}
We decompose ${K_{\alpha}}(\rho, x, \rho' ,x')$ into two parts,
\begin{align*}
    K_{\alpha}^1(\rho, x, \rho' ,x')&=\frac{1}{\Gamma(\alpha)}\int_0^1  t^{\alpha-1} E(t,\rho, x, \rho' ,x') \dd t, \\
    K_{\alpha}^2(\rho,x, \rho', x')&=\frac{1}{\Gamma(\alpha)}\int_1^\infty  t^{\alpha-1}E(t,\rho,x,\rho',x') \dd t.
\end{align*}

We firstly estimate the term $ K_{\alpha}^2(\rho,x, \rho', x')$. Together   the  inequality that 
  $
  2\coth 2t -\tanh t >\coth 2t > 1
  $,  with the fact that $$\tanh t \sim 1, \; \sinh 2t \sim \ee^{2t },\; \coth 2t \sim \ee^{2t},\;\text{as}\;  t \to \infty,$$ 
 we have
  \begin{align*}
      |K_{\alpha}^2(\rho,x, \rho', x')|&\leq C \ee^{-|x-x'|^2-(\rho-\rho')^2-|x+x'|^2} \int_1^\infty t^{\alpha-3/2}\ee^{-td}\dd t\\
      &\leq C \ee^{-|x-x'|^2-(\rho-\rho')^2-|x+x'|^2}.
  \end{align*}

We next estimate $K_{\alpha}^1(\rho, x, \rho' ,x')$. We further split it into two cases.

  \emph{Case 1:} $(z, z') \in D_{+}\coloneq\{(z, z') \in \mathbb R^{d+1} \times \mathbb R^{d+1}; \; |z-z'|\ge1\}$. Using the fact that $\sinh 2t \sim 2t$, $\coth 2t \sim \frac{1}{2t}$ and $\tanh t\sim t$ as $t\to 0$, we have
  \begin{align*}
    |K_{\alpha}^1(\rho, x, \rho' ,x')|&\lesssim  \int_0^1 t^{\alpha-1-\frac{d+1}{2}} \ee^{-\frac{1}{8t}[|x-x'|^2+(\rho-\rho')^2]} \dd t\\
    &\lesssim \ee^{-\frac{1}{16}[|x-x'|^2+(\rho-\rho')^2]} \int_{0}^1 t^{\alpha-1-\frac{d+1}{2}} \ee^{-\frac{1}{16 t}} \dd t\\
  & \lesssim \ee^{-\frac{1}{16}[|x-x'|^2+(\rho-\rho')^2]}.
  \end{align*}

\emph{Case 2:} $(z, z') \in D_{-}\coloneq \{(z, z') \in \mathbb R^{d+1} \times \mathbb R^{d+1}; \; |z-z'|<1\}$. If  $\alpha<\frac{d+1}{2}$, we  have
\begin{align*}
     |K_{\alpha}^1(\rho, x, \rho' ,x')|&\leq  \int_0^1 t^{\alpha-1-\frac{d+1}{2}} \ee^{-\frac{1}{8t}[|x-x'|^2+(\rho-\rho')^2]} \dd t\\
     &\lesssim \frac{1}{[|x-x'|^2+(\rho-\rho')^2]^{(d+1)/2-\alpha}};
\end{align*}
if $\alpha=\frac{d+1}{2}$, we have
\begin{align*}
      |K_{\alpha}^1(\rho, ,x,  \rho', x')| \lesssim  \log[|x-x'|^2+(\rho-\rho')^2];
\end{align*}
and finally, if $\alpha>\frac{d+1}{2}$, we have
\begin{align*}
    |K_{\alpha}^1(\rho, x, \rho' ,x')|\lesssim 1.
\end{align*}

It follows that  \eqref{kernel-ineq} holds with the integrable function $\Psi_\alpha$ defined by
\begin{align}\label{control function}
    \Psi_\alpha(z-z')=\begin{cases}
       \mathbf 1_{D_-}|z-z'|^{2\alpha-(d+1)}+ \mathbf 1_{D_+}\ee^{-\frac{1}{16}|z-z'|^2}, & \alpha<\frac{d+1}{2},\\
       \mathbf 1_{D_-} \log|z-z'|+ \mathbf 1_{D_+}\ee^{-\frac{1}{16}|z-z'|^2}, & \alpha=\frac{d+1}{2},\\
        \mathbf 1_{D_-} + \mathbf 1_{D_+}\ee^{-\frac{1}{16}|z-z'|^2}, & \alpha>\frac{d+1}{2},
    \end{cases}
\end{align}
which completes the proof.
\end{proof}

By \eqref{kernel of H+a},  we can obtain  similar estimates for the  integral  kernels of the operators $(\AH+2)^{-\alpha}$ and $(\AH-2)^{-\alpha}$:
\begin{cor}\label{kernel of H+-2} Assume that $\alpha>0$. Let $M_\alpha (\rho, x, \rho' ,x')$ and $N_\alpha (\rho, x, \rho' ,x')$ be the integral kernels of the operators $(\AH+2)^{-\alpha}$ and $(\AH-2)^{-\alpha}$ respectively. Then, we have
\begin{align*}
    M_\alpha (\rho, x, \rho' ,x') \leq  {K_{\alpha}}(\rho, x, \rho' ,x')\leq   \Psi_\alpha(z-z'),
\end{align*}
and
\begin{align*}
   N_\alpha (\rho, x, \rho' ,x') \leq   \Psi_\alpha(z-z'), \quad  \text{if}\;\;  d\geq  3.
\end{align*}
\end{cor}

\begin{proof} By \eqref{kernel of H+a}, the  estimate of $M_\alpha $  is trivial. It suffices to show the estimate of $N_\alpha$. By definition, we have
\begin{align}
     N_\alpha (\rho, x, \rho' ,x')=   \frac{1}{\Gamma(\alpha)}\int_0^\infty  t^{\alpha-1} e^{2t} E(t,\rho, x, \rho' ,x') \dd t.
\end{align}
We once again decompose $N_\alpha (\rho, x, \rho' ,x')$ into two parts as follows.
 \begin{align*}
    N_{\alpha}^1(\rho, x, \rho' ,x')&=\frac{1}{\Gamma(\alpha)}\int_0^1  t^{\alpha-1} e^{2t}E(t,\rho, x, \rho' ,x') \dd t, \\
    N_{\alpha}^2(\rho,x, \rho', x')&=\frac{1}{\Gamma(\alpha)}\int_1^\infty  t^{\alpha-1}e^{2t}E(t,\rho,x,\rho',x') \dd t.
\end{align*}

For the term $ N_{\alpha}^1(\rho, x, \rho' ,x')$,  we can proceed by the same way as the estimate of $K_{\alpha}^1(\rho, x, \rho' ,x')$, we omit the details here. For  $N_{\alpha}^1(\rho, x, \rho' ,x')$ term, we obtain for  $d\geq 3$ that
 \begin{align*}
      |N_{\alpha}^2(\rho,x, \rho', x')|&\leq C \ee^{-|x-x'|^2-(\rho-\rho')^2-|x+x'|^2} \int_1^\infty t^{\alpha-3/2}\ee^{-td}e^{2t}\dd t\\
      &\leq C \ee^{-|x-x'|^2-(\rho-\rho')^2-|x+x'|^2},
  \end{align*}
which completes the proof. \end{proof}

\begin{rem}We make some comments about the condition that $d > - a=2$ for the second result of $N_{\alpha}^2(\rho,x, \rho', x')$ in the above corollary. Since the integral $\int_1^\infty t^{\alpha-3/2}e^{2t}e^{-t} \dd t$ is divergent, we cannot obtain similar estimates for $(\AH-2)^{-\alpha}$ for the case $d=1$. It is consistent with the spectral property that $\sigma(\AH-2)=[-1, \infty)$ when $d=1$.  When $d=2$, the spectral property is  that $\sigma(\AH-2)=[0, \infty)$, and we have different behavior depending on the power $\alpha$. If $0<\alpha<1/2$,   the kernel of $(\AH-2)^{-\alpha}$ has  exponential decay as $|z-z'| \to \infty$ since $\int_1^\infty t^{\alpha-3/2}e^{2t}e^{-2t} \dd t$ is convergent, and if  $\alpha\geq 1/2$, no useful result can be derived.
\end{rem}

In addition, we also have the following  boundedness result for the negative fractional powers of the operator $\AH$.
 \begin{prop}\label{decay in x}
  Let $p\in [1, \infty]$ and $\alpha>0$. Then the weighted operator
  \begin{align*}
      |x|^{2\alpha} \AH^{-\alpha}
  \end{align*}
is bounded on $L^p(\mathbb R^{d+1})$.
\end{prop}
\begin{proof} By Schur's lemma \cite{Gra250}, we only need to verify that
\begin{align}
    \sup_{z} |x|^{2\alpha} \int_{\mathbb R^{d+1}} |{K_{\alpha}}(z, z')| \dd z' \lesssim 1,  \label{decay in x-1}
    \end{align}
    and
    \begin{align}
     \sup_{z'} \int_{\mathbb R^{d+1}}  |x|^{2\alpha} |{K_{\alpha}}(z, z')| \dd z \lesssim 1.\label{decay in x-2}
\end{align}

By the boundedness in Proposition \ref{prop31}, we may  assume $|x|\geq 2$.

Firstly, we prove \eqref{decay in x-1}. We partition $\mathbb R^{d}=E_x \cup E_x^c$, where
\begin{align*}
    E_x=\{x'\in \mathbb R^{d}: |x|>2|x-x'|\}.
\end{align*}

     When $x'\in E^c_x$, i.e., $|x|\le 2|x-x'|$.  By the fact that
\[
         2\coth2t-\tanh t \geq \frac{1}{2}(2\coth2t-\tanh t)+\frac{1}{4},
\]
we have
\begin{align*}
  \quad  K_\alpha(\rho, x,  \rho', x')
   & \lesssim \ee^{-\frac{1}{16}|x-x'|^2}\int_0^\infty t^{\alpha-1-1/2} (\sinh 2t)^{-d/2} \ee^{-\frac{1}{2}(B(t, \rho, \frac{x}{\sqrt 2}, \rho',  \frac{x'}{\sqrt 2} )}  \dd t\\
   & \lesssim \ee^{-\frac{1}{16}|x-x'|^2} {K_{\alpha}}\left(\rho, \frac{x}{\sqrt 2}, \rho', \frac{x'}{\sqrt 2} \right).
\end{align*}
So it follows that
\begin{align*}
  & \quad |x|^{2\alpha}  \int_\mathbb R \int_{E_x^c} |{K_{\alpha}}(\rho, x,  \rho', x') |\dd x' \dd \rho'\\
  &\lesssim \int_{\mathbb R^{d+1}}\underbrace{ |x-x'|^{2\alpha} \ee^{-\frac{1}{16}|x-x'|^2}}_{<\; \infty}\Big|{K_{\alpha}}\Big(\rho, \frac{x}{\sqrt 2}, \rho', \frac{x'}{\sqrt 2} \Big) \Big|\dd x' \dd \rho'\\
  &\lesssim 1.
\end{align*}

When $x' \in E_x$, i.e., $|x|> 2|x-x'|$, we once again decompose  $K_\alpha=K^1_\alpha+K^2_\alpha$ as in the proof of Proposition \ref{prop31}. For $K^2_{\alpha}(z, z')$ term. Since $|x|>2|x-x'|$ implies  $|x|<|x+x'|$, we have
\begin{align*}
    &\quad |x|^{2\alpha}\int_{\mathbb R} \int_{E_x} |x|^{2\alpha}\ee^{-\frac{1}{8}|x+x'|^2} \ee^{-\frac{1}{4}|x-x'|^2} \ee^{-(\rho-\rho')^2} \dd x'\dd \rho'\\
    &\leq \int_{\mathbb R} \int_{\mathbb R^d} \underbrace{|x|^{2\alpha}\ee^{-\frac{1}{8}|x|^2}}_{\lesssim\; 1} \ee^{-\frac{1}{4}|x-x'|^2} \ee^{-(\rho-\rho')^2} \dd x'\dd \rho'\\
    &\lesssim 1.
\end{align*}
As for the $K_\alpha^1 (z, z')$ term,  we have

\begin{align*}
   &\int_{\mathbb R} \int_{E_x} |x|^{2\alpha} \int_0^1 t^{\alpha-\frac{d+3}{2}} \ee^{-\frac{1}{8 t}|x-x'|^2-\frac{1}{4}t|x|^2 -\frac{(\rho-\rho')^2}{4t}} \dd t \dd x' \dd \rho'\\
   &\leq |x| \int_{\mathbb R} \int_{0}^{|x|^2} \int_0^{|x|^2}  u^{\alpha-\frac{d+3}{2}} \ee^{-\frac{|x|^2}{8 u}(\rho-\rho')^2} \ee^{-\frac{r^2}{4 u} +\frac{1}{4}u} \dd u  r^d \dd r\dd \rho'\\
   &\leq |x| \int_{0}^{|x|^2} \int_0^{|x|^2}  u^{\alpha-\frac{d+3}{2}}\underbrace{ \int_{\mathbb R} \ee^{-\frac{|x|^2}{8 u}(\rho-\rho')^2} \dd \rho'}_{\lesssim\; u^{1/2}|x|^{-1}}\ee^{-\frac{r^2}{4 u} +\frac{1}{4}u} \dd u  r^d \dd r\\
   &\lesssim  \int_{0}^{|x|^2} \int_0^{|x|^2}  u^{\alpha-\frac{d+2}{2}}  \ee^{-\frac{r^2}{4 u} +\frac{1}{4}u} \dd u \,  r^d \dd r <\infty.
\end{align*}
This completes the proof of \eqref{decay in x-1}.

 To prove \eqref{decay in x-2}, we similarly partition $\mathbb R^{d}=E_{x'}\cup E_{x'}^c$, with
\begin{align*}
    E_{x'}=\{x\in \mathbb R^d: |x'|\geq 2|x-x'|\},
\end{align*}
and write
\begin{align*}
 \int_{\mathbb R^{d+1}} |x|^{2\alpha}{K_{\alpha}}(z, z') \dd z=\bigg(\int_{\mathbb R} \int_{E_{x'}}+\int_{\mathbb R} \int_{E_{x'}^c}\bigg)|x|^{2\alpha}{K_{\alpha}}(\rho, x,  \rho', x') \dd x \dd \rho.
\end{align*}

   When $x\in E_{x'}$, i.e., $|x'| \geq 2|x-x'|$, we have $|x| \leq \frac{3}{2}|x'|$. Since the kernel is symmetric in $z$ and $z'$, that is, ${K_{\alpha}}(\rho, x,  \rho', x')={K_{\alpha}}(\rho', x', \rho, x)$,  we have
 \begin{align*}
  \int_{\mathbb R} \int_{E_{x'}}|x|^{2\alpha}{K_{\alpha}}(\rho, x,  \rho', x') \dd x \dd \rho
 \lesssim   |x'|^{2\alpha} \int_{\mathbb R}  \int_{|x'|>2|x-x'|} |{K_{\alpha}}(\rho', x', \rho, x) |\dd x \dd \rho.
 \end{align*}
By \eqref{decay in x-1},  the right hand side of the above inequality is finite.  

When $x\in E^c_{x'}$, we further split the integral domain into two cases. When $x\in E^c_{x'}$ and $|x-x'|<1$,  it follows that $|x|\leq |x'|+|x-x'| \leq 4$. Hence,
 \begin{align*}
\int_{\mathbb R} \int_{|x'|<2|x-x'|}|x|^{2\alpha}|{K_{\alpha}}(\rho, x,  \rho', x')| \dd x \dd \rho
     &\lesssim \int_{\mathbb R} \int_{|x'|<2|x-x'|}|{K_{\alpha}}(\rho, x,  \rho', x')| \dd x \dd \rho \\
     &\lesssim 1.
 \end{align*}
  When $x\in E^c_{x'}$ and $|x-x'|>1$, we have
\begin{align*}
    &\quad\int_{\mathbb R} \int_{\{|x'|<2|x-x'|, \ |x-x'|\geq 1\}}|x|^{2\alpha}|{K_{\alpha}}(\rho, x,  \rho', x')| \dd x \dd \rho\\
    &\lesssim \int_{\mathbb R} \int_{|x-x'|\geq 1}|x-x'|^{2\alpha}|{K_{\alpha}}(\rho, x,  \rho', x') |\dd x \dd \rho\\
    &\lesssim \int_{\mathbb R^{d+1}}\underbrace{ |x-x'|^{2\alpha} \ee^{-\frac{1}{16}|x-x'|^2}}_{\lesssim\;  1}\Big|{K_{\alpha}}\Big({\rho}, \frac{x}{\sqrt 2}, \rho', \frac{x'
    }{\sqrt 2} \Big) \Big|\dd x' \dd \rho' \\
    &\lesssim 1.
\end{align*}
This proves \eqref{decay in x-2} and completes the proof.
 \end{proof}
\subsection{Symbols for fractional powers $\AH^{\alpha}$}
Using the heat kernel representation of the  fractional powers operator $\AH^{\alpha}$ in Subsection \ref{sub: heat kernel rep},  we firstly calculate the symbol of the heat semigroup $\ee^{-t \AH}$.
Together  \eqref{heat-formula-1}, the fact that $\widehat \Phi_{\mu}=(-i)^{|\mu|} \Phi_{\mu}$, with Plancherel's theorem, we have
\begin{align}
      &\ee^{-t \AH}f(\rho, x)\notag
      \\
      &=\sum_{\mu} \frac{1}{\sqrt{2\pi}}\int_{\mathbb R} \ee^{i\tau \rho} \ee^{-\tau^2 t} \ee^{-(2|\mu|+d)t} \ee^{-\frac{\pi i}{4}2|\mu|}(\mathcal F_{\rho, x}f(\tau, \xi), \Phi_{\mu}(\xi))\Phi_{\mu}(x) \dd \tau \notag
      \\
      &=\frac{1}{(2\pi)^{(d+1)/2}}\int_{\mathbb R^{d+1}} \ee^{i\tau \rho} \ee^{i\xi\cdot x} p_t(\rho, x, \tau, \xi) \mathcal F_{\rho, x}f(\tau, \xi) \dd \tau \dd \xi \label{symbol of semigroup},
\end{align}
where
 \begin{align*}
     p_t(\rho, x, \tau, \xi)=\sum_{\mu} \ee^{-i \xi \cdot x}\ee^{-\tau^2 t} \ee^{-(2|\mu|+d)t} \ee^{-\frac{\pi i}{4}2|\mu|} \Phi_{\mu}(\xi) \Phi_{\mu}(x).
 \end{align*}
In view of Mehler's formula \eqref{Mehler's formula},   we  have
\begin{align*}
    p_t(\rho, x, \tau, \xi)= c_d (\cosh 2t )^{-\frac{d}{2}} \ee^{-b(t, x, \tau, \xi )},
\end{align*}
where $c_d={(2\pi)^{-d/2}}$ and
\begin{align}\label{form of b}
b(t, x, \tau, \xi ) &=\frac{1}{2} (|x|^2+|\xi|^2)\tanh 2t +2i x  \cdot \xi \sech 2t (\sinh t)^2 +t \tau^2.
\end{align}


Now we have
\begin{lem}\label{symbol of fractional powers}
    Let $\alpha\in \mathbb R$. The symbol $\sigma_{\alpha}(\rho, x, \tau, \xi)$ of the operator $\AH^{\alpha}$ belongs to the symbol class $G^{2\alpha}$ defined by Definition \ref{generalized symbol}.
    \end{lem}
\begin{proof}
From the explicit formula of $p_t$, we know that $\sigma_\alpha$ does not depend on $\rho$. Thus, the estimates for $\sigma_\alpha$ do not depend on $\rho$ and its derivatives on $\rho$ are zero. For brevity, we will write  $\sigma_{\alpha}( x, \tau, \xi)$ instead of $\sigma_{\alpha}(\rho, x, \tau, \xi)$, and similarly for other terms which are independent of $\rho$.

The case $\alpha=0$ is obvious. 

For the case $\alpha<0$. Using the equality of \eqref{negative power and semigroup}, we have
\begin{align*}
   \sigma_{\alpha}( x, \tau, \xi)=\frac{1}{ \Gamma(-\alpha)} \int_0^\infty t^{-\alpha-1} p_t( x, \tau, \xi) \dd t.
\end{align*}
We split $ \sigma_{\alpha}( x, \tau, \xi)$ into two parts as follows.
\begin{align*}
    \sigma_{\alpha}^1( x, \tau, \xi)=\frac{1}{ \Gamma(-\alpha)} \int_0^1 t^{-\alpha-1} p_t( x, \tau, \xi) \dd t,\\
   \sigma_{\alpha}^2( x, \tau, \xi)=\frac{1}{ \Gamma(-\alpha)} \int_1^\infty t^{-\alpha-1} p_t( x, \tau, \xi) \dd t.
\end{align*}

To estimate $\sigma_{\alpha}^1$, on the one hand, by the lower bound estimate for the real part of $b$ as $t\in(0, 1)$,
\begin{align*}
 \Re b(t, x, \tau, \xi ) \geq  c t (|x|^2+|\xi|^2+\tau^2),
\end{align*}
we have
\begin{align*}
   |\sigma_{\alpha}^1( x, \tau, \xi)|\leq\;  \frac{1}{ \Gamma(-\alpha)} \int_0^1 t^{-\alpha-1} e^{- c(|x|^2+|\xi|^2+\tau^2)  t} \dd t\;
   &\leq (1+|x|^2+|\xi|^2+\tau^2)^\alpha.
\end{align*}

On the other hand, by  \eqref{form of b},  the derivatives of  $b(t,  x, \tau, \xi)$ satisfy 
\begin{align*}
    \partial_{x_j} b(t,  x, \tau, \xi)&=x_j \tanh 2t +2i \xi_j\sech 2t (\sinh t)^2 \sim x_j t + \xi_jt^2,\\
    \partial_{x_j} ^2b(t,  x, \tau, \xi)&= \tanh 2t \sim t,\\
     \partial_{\xi_j} b(t,  x, \tau, \xi)&= \xi_j \tanh 2t +2i x_j\sech 2t (\sinh t)^2 \sim \xi_jt + x_jt^2,\\
     \partial_{\xi_j} ^2b(t,  x, \tau, \xi)&= \tanh 2t \sim t,\\
     \partial_{\tau}b(t,  x, \tau, \xi)&= 2\tau t, \quad \partial_{\tau}^2b(t,  x, \tau, \xi)= 2 t,
\end{align*}
as $t\to 0$.  In conclusion, we obtain with the shorthand $X=(x,\tau,\xi)$ that 
\begin{align*}
   |\del_X^\beta b(t,  x, \tau, \xi) |   \lesssim \begin{cases}
    |X|^{2-|\beta|}t,  &|\beta|  \leq 2, \\
       0, & |\beta| \geq 3.
   \end{cases}
\end{align*}
By using Fa\`a di Bruno's formula, we obtain 
  \begin{align*}
    \partial_X^{\beta} \sigma_{\alpha}^1( x, \tau, \xi)=\frac{1}{ \Gamma(-\alpha)} \int_0^1 t^{-\alpha-1} p_t( x, \tau, \xi) \prod_{\substack{1\le j \le |\beta| \\\sum_j|\beta_j|n_j=|\beta|}} ( \partial_{X}^{\beta_j}b)^{n_j}\dd t. 
\end{align*}
Hence, for any $|\beta| \geq 1$, we get
\begin{align*}
    |\partial_X^{\beta} \sigma_{\alpha}^1( x, \tau, \xi)|&\lesssim \int_0^1 t^{-\alpha-1} e^{-c t |X|^2} \prod_{j=1}^{|\beta|}t^{n_j} \dd t \prod_{\substack{1\le j \le |\beta| \\\sum_j|\beta_j|n_j=|\beta|}}  |X|^{2n_j-|\beta_j|n_j}\\
    &\lesssim \ip{ X}^{2\alpha-|\beta|}.
\end{align*}
That is, $\sigma_{\alpha}^1 \in G^{2\alpha}$ for all $\alpha<0$.

To estimate $\sigma_{\alpha}^2$, since  $\cosh t \sim e^{t}$, and $\tanh t \geq t$ as $t\to\infty$,  we have
\begin{align*}
    \Re b(t,  x, \tau, \xi) \geq  ct |X|^2,
\end{align*}
which implies that 
\begin{align*}
   | \sigma_{\alpha}^2( x, \tau, \xi)| \lesssim  \int_1^\infty t^{-\alpha-1}  e^{-td } e^{-c |X|^2} \dd t 
    \lesssim e^{-|X|^2}.
\end{align*}
When we take partial derivatives in $X=(x,\tau,\xi)$, we only change the degree of the polynomials in $X$. The dominating term is still  $e^{-|X|^2}.$ Hence, we have  $\sigma_\alpha^2 \in G^{\alpha}$.

For the case $\alpha>0$.  By \eqref{positive power and semigroup}, we obtain the symbol of the operator $\AH^\alpha$,
\begin{align*}
   \sigma_{\alpha}( x, \tau, \xi)=\frac{(-1)^N}{\Gamma(N-\alpha)}\int_{0}^\infty t^{N-\alpha-1} \frac{\dd^N }{\dd t^N} p_t( x, \tau, \xi)\dd t.
\end{align*}
Arguing as in the case $\alpha<0$, we split it into two parts
\begin{align*}
     \sigma_{\alpha}^1( x, \tau, \xi)=\frac{(-1)^N}{\Gamma(N-\alpha)}\int_{0}^1t^{N-\alpha-1} \frac{\dd^N }{\dd t^N} p_t( x, \tau, \xi)\dd t,\\
     \sigma_{\alpha}^1( x, \tau, \xi)=\frac{(-1)^N}{\Gamma(N-\alpha)}\int_{1}^\infty t^{N-\alpha-1} \frac{\dd^N }{\dd t^N} p_t( x, \tau, \xi)\dd t.
\end{align*}
Note that
\begin{align*}
    &\quad\partial_t b (t, x, \tau, \xi)\\
    &= 2 (|x|^2+|\xi|^2) \sech^2 2t+\tau^2+ 2ix\cdot \xi [ -2 \sech^2 2t \sinh 2t \sinh^2 t+\tanh 2t] \\
&= 2 (|x|^2+|\xi|^2) \sech^2 2t+\tau^2 + 2ix\cdot \xi \sech^2 2t \sinh 2t.
\end{align*}
Thus, for $t\in(0,1)$ we have 
\begin{align*}
|\partial_t (\cosh 2t )^{-d/2}| \lesssim 1,   \quad  |\partial_t b (t, x, \tau, \xi)| \lesssim |X|^2,
\end{align*}
and
\begin{align*}
|\partial_t^N (\cosh 2t )^{-d/2}|&\lesssim 1,    \quad |\partial_t ^N b (t, x, \tau, \xi) | \lesssim |X|^{2N}.
\end{align*}

To compute the derivative estimates of $p_t( x, \tau, \xi)$ in the term $ \sigma_{\alpha}^1( x, \tau, \xi)$, we use the above estimates, Leibniz rule and Fa\`a di Bruno's formula, and obtain that 
\begin{align*}
   | \sigma_{\alpha}^1( x, \tau, \xi)| &\lesssim\int_{0}^1 t^{N-\alpha-1} e^{-c|X|^2} |X|^{2N}\dd t 
     \lesssim \ip{ X}^{2\alpha}.
\end{align*}
For the derivative estimates, we have
\begin{align*}
     |\partial_X^{\beta} \sigma_{\alpha}^1( x, \tau, \xi)|&\lesssim \int_0^1 t^{N-\alpha-1} e^{-c t |X|^2} \prod_{\substack{1\le j \le |\beta| \\\sum_j|\beta_j|n_j=|\beta|}}|X|^{2N}  |X|^{2n_j-|\beta_j|n_j} \\
    &\lesssim  \ip{ X}^{2\alpha-|\beta|}.
\end{align*}
For the term $\sigma_{\alpha}^2$, the exponential decay can be obtained as the case $\alpha<0$.

Summing up,  we have $\sigma_\alpha \in G^{2\alpha}$ for all $\alpha \in \mathbb R$, and this concludes the proof.
%
\end{proof}

\subsection{Riesz transforms and symbols}
The $j$th Riesz transforms associated with the operator $ \AH$ are defined as
\begin{align}\label{Riesz transform}
R_j= A_j\AH^{-1/2}, \quad -d\le j\le d.
\end{align}
In general, for any $m\in \mathbb N$ and $\mathbf j=(j_{1},\dots,j_m)$,  $-d \leq j_l\leq d$, the $\mathbf j$th Riesz transform of order $m$ is the operator
 \begin{align}\label{m order Riesz transform}
     R_{\mathbf j} = R_{j_1,\dots,j_m} &=A_{j_1}A_{j_2}\dots A_{j_m} \AH^{-m/2}\nonumber\\
     &=P_m(\partial_\rho, \partial_{x}, x) \AH^{-m/2},
 \end{align}
 where $P_m$ is a polynomial of degree $m$.
 
In this subsection, we will prove that the Riesz transforms  defined by \eqref{Riesz transform} and \eqref{m order Riesz transform} are bounded on classical Sobolev spaces by verifying that their symbols belong to the symbol class $S^0_{1,0}$. There are two ways to show this. The most obvious way would be a direct calculation by the heat semigroup. The symbol of  the operator $A_0 \ee^{-t \AH}$ is
$-i\tau p_t( x, \tau, \xi)$, so the symbol of Riesz transform $R_0$ is
\begin{align} \label{symbol of R0}
    \sigma_{R_0}(\rho, x, \tau, \xi)=\frac{-1}{\sqrt{\pi}} \int_0^\infty t^{-\frac{1}{2}}i\tau p_t( x, \tau, \xi) \dd t.
\end{align}
For the symbol of the operator $A_j \ee^{-t \AH}$, $1\leq j\leq d$, we use the formula \eqref{symbol of semigroup} to get
\begin{align*}
     A_j \ee^{-t \AH} f(\rho, x)
    =\int_{\mathbb R^{d+1}} \ee^{i\tau \rho} \ee^{i\xi\cdot x} p_t( x, \tau, \xi)(-i\xi_j+x_j+\del_{x_j} b ) \widehat f(\tau, \xi) \dd \tau \dd \xi.
\end{align*}
This shows that the symbol of the operator $A_j \ee^{-t \AH}$ is
\begin{align*}
    p_t( x, \tau, \xi)(-i\xi_j+x_j+\del_{x_j} b(t,\rho,x,\tau,\xi)).
\end{align*}
So, the symbols  for  Riesz transforms $R_j$  for $1\leq j \leq d$ are
\begin{align} \label{symbol of R_j}
     \quad \sigma_{R_j}(\rho, x, \tau, \xi)
    =\frac{-1}{\sqrt{\pi}} \int_0^\infty t^{-\frac{1}{2}}(i\xi_j-x_j+\del_{x_j} b) p_t( x, \tau, \xi) \dd \tau.
\end{align}
Similar calculations give the symbols for Riesz transforms $R_j $ for $-d \leq j \leq -1$.
It is then possible to use the integral forms  \eqref{symbol of R0} and \eqref{symbol of R_j} in a lengthy calculation like in the proof of Lemma \ref{symbol of fractional powers} to show that they belong to the symbol class $S^0_{1,0}$.

The second, simpler and shorter way is to take advantage of the symbol calculus for compositions in $G^m$, which we present below
\begin{prop}\label{prop:RT bd} The symbols $\sigma_{R_j}$  of Riesz transforms $R_j$ for $0\leq|j |\leq d$ belongs to the symbol class $S^0_{1,0}$, hence they are bounded on classical Sobolev spaces $W^{\alpha,p}(\mathbb R^{d+1})$ for any $\alpha \in \mathbb R$ and $1<p<\infty$. In addition, the same result holds for Riesz transforms $R_{\mathbf j}$ of high order.
\end{prop}
\begin{proof} The symbols of the operators $A_j$  are given by either $i\tau$ or $\pm i\xi_j+x_j $,  which belong to class $G^{1}$. From Proposition \ref{symbol of fractional powers},  the symbol of the operator $\AH^{-1/2}$ belongs to class $G^{-1}$.  By symbolic calculus in class $G^m$ (Proposition \ref{composition}), we obtain that symbols of Riesz transform $R_j$ for $0\leq|j |\leq d$  belong to class $S_{1,0}^{0}$, which implies the boundedness of  Riesz transform $R_j$ for $0\leq|j |\leq d$  on $W^{\alpha, p}$ for all $1<p<\infty$, (see \cite{Taylor, Taylor91}).

For Riesz transform $R_{\mathbf j}$ with high order, By Proposition \ref{composition} and the fact that the symbols of the operators $A_{j_1}A_{j_2}\dots A_{j_m} $ and $\AH^{-m/2}$ belongs to symbol classes $G^{m}$  and $G^{-m}$ respectively. Hence, the symbol of $R_{\mathbf j}$ belongs to the symbol the symbol class $S_{1,0}^0$, which implies the  result and completes the proof.
\end{proof}


 \section{Sobolev spaces associated to the partial harmonic oscillator}
\label{s:4-sobolev-spaces}
 Given any $p\in [1, \infty)$, and $\alpha>0$, we define the potential spaces associated to $ \AH$ by
\begin{align*}
    L_{ \AH}^{\alpha, p}(\mathbb R^{d+1}) = \AH^{-\alpha/2} (L^p(\mathbb R^{d+1})),
\end{align*}
 with the norm
 \begin{align*}
     \|f\|_{L_{\AH}^{\alpha, p}}  =\| g\|_{L^p(\mathbb R^{d+1})},
 \end{align*}
 where $g\in L^p(\mathbb R^{d+1})$ satisfies  $\AH^{-\alpha/2}g=f$.
\begin{rem} The norm is well defined since $ \AH^{-\alpha/2}$ is one-to-one and   bounded in $L^p(\mathbb R^{d+1})$.
Also, $C_0^\infty(\mathbb R^{d+1}) $ is dense in $L_{\AH}^{\alpha, p}(\mathbb R^{d+1}) $.
\end{rem}

For any nonnegative integer $k\geq 0$, we  can also define the Sobolev spaces associated to $ \AH$ by the differential operators $A_j$ as follows:
\begin{align*}
    W_{\AH}^{k, p}=\left\{f\in L^p(\mathbb R^{d+1}) \middle| \substack{ \displaystyle A_{j_1}A_{j_2}\dots A_{j_m} f \in L^p(\mathbb R^{d+1}), \\[0.5em]\displaystyle \text{for any}\;\; 
   1\leq m \leq k, \  0\leq |j_1|, \dots, |j_m| \leq d}\right\},
\end{align*}
with the norm
\begin{align*}
   \|f\|_{W_{\AH}^{k, p}}=\sum_{m=1}^k\Bigg(\sum_{j_1=-d}^d
   \dots \sum_{j_m=-d}^d\|A_{j_1}A_{j_2}\dots A_{j_m} f\|_{L^p}\Bigg)+\|f\|_{L^p}.
\end{align*}


\begin{thm} Let $k \in \mathbb N$ and $p\in(1, \infty)$. Then  we have
\begin{align*}
   W_{\AH}^{k, p}(\mathbb R^{d+1})=     L_{\AH}^{k, p}(\mathbb R^{d+1})
\end{align*}
with equivalence of norms.
\end{thm}
\begin{proof} We firstly prove that $  L_{\AH}^{k, p}(\mathbb R^{d+1})\subset W_{\AH}^{k, p}(\mathbb R^{d+1})$. For any function $f\in L_{\AH}^{k, p}(\mathbb R^{d+1})$, there exists a function $ g\in L^{p}(\mathbb R^{d+1})$, such that $f=\AH^{-k/2} g$. Hence, by  the $L^p$ boundedness of Riesz transforms in Proposition \ref{prop:RT bd}, we have
\begin{align*}
\|A_{j_1}A_{j_2}\dots A_{j_m} f \|_p = \|A_{j_1}A_{j_2}\dots A_{j_m} \AH^{-k/2} g\|_{p} 
   \lesssim \|g\|_{p} \leq C \|f\|_{L_{\AH}^{k, p}(\mathbb R^{d+1})},
\end{align*}
which implies that $  L_{\AH}^{k, p}(\mathbb R^{d+1})\subset W_{\AH}^{k, p}(\mathbb R^{d+1})$.

Next, we show that   $ W_{\AH}^{k, p}(\mathbb R^{d+1})\subset L_{\AH}^{k, p}(\mathbb R^{d+1})$ by induction.
First, it is easy to check that for any $f, g\in C_0^\infty(\mathbb R^{d+1})$, we have
\begin{align}\label{duality identity}
   \int_{\mathbb R^{d+1}}  f g =2 \int_{\mathbb R^{d+1}} \sum_{-d\le j\le d} R_j f R_j g.
\end{align}
Thus, by duality and  the boundedness of Riesz transform, we obtain for any $g\in L^p(\mathbb R^d)$ that
\begin{align*}
    \|g\|_p \lesssim \sum_{-d\le j\le d} \|R_j g\|_p.
\end{align*}
Hence, we obtain by choosing $g=\AH^{1/2}f$ that 
\begin{align}\label{inverse Riesz transform}
    \|f\|_{L_{\AH}^{1, p}(\mathbb R^{d+1})} =\|\AH^{1/2}f\|_p\lesssim \sum_{-d\le j\le d} \|A_j f\|_p \lesssim \|f\|_{W_{\AH}^{1,p}(\mathbb R^{d+1})}.
\end{align}
That is, $ W_{\AH}^{k, p}(\mathbb R^{d+1})\subset L_{\AH}^{k, p}(\mathbb R^{d+1})$ for $k=1$.

Suppose that for any $f\in C_0^\infty(\mathbb R^{d+1})$ and any  $1\leq m<k$, we have
\begin{align}\label{ind assump}
     \|f\|_{L_{\AH}^{m, p}(\mathbb R^{d+1})}\leq \|f\|_{W^{m,p}_{H_{par}}(\mathbb R^{d+1})}.
\end{align}
It follows by duality that
\begin{align*}
   \|f\|_{L_{\AH}^{k, p}(\mathbb R^{d+1})}&=\|\AH^{k/2}f\|_{L^p(\mathbb R^{d+1})}  =\sup_{g\in C_0^\infty; \|g\|_{p'}=1} \int_{\mathbb R^{d+1}}\AH^{k/2}f g\dd z\\
   &=\sup_{g\in C_0^\infty: \|g\|_{p'}=1} \int_{\mathbb R^{d+1}}\AH^{k}f  \AH^{-k/2}g\dd z.
\end{align*}
Since there exist constants $c_1, c_2, \dots , c_{k-1}$ such that
\begin{align*}
\sum_{0\leq |j_1|, \dots, |j_k| \leq d} A_{j_k}^*\dots A_{j_1}^* A_{j_1}\dots A_{j_k} = 2^k \AH^{k}+\sum_{m=1}^{k-1} c_m \AH^m,
\end{align*}
we obtain  that
\begin{align*}
   &\quad 2^k  \int_{\mathbb R^{d+1}}\AH^{k}f  \AH^{-k/2}g\dd z\\
    &= \int_{\mathbb R^{d+1}} \sum_{0\leq |j_1|, \dots, |j_k| \leq d}\left( A_{j_k}^*\dots A_{j_1}^* A_{j_1}\dots A_{j_k}-\sum_{m=1}^{k-1} c_m \AH^m\right) f \AH^{-k/2} g\\
    &=\!\sum_{0\leq |j_1|, \dots, |j_k| \leq d} \int_{\mathbb R^{d+1}}\!\!A_{j_1}\dots A_{j_k} f R_{j_1}\dots R_{j_k} g-\sum_{m=1}^{k-1} c_m \int_{\mathbb R^{d+1}} \AH^{m/2} f\AH^{-\frac{k-m}{2}} g\\
    &\leq  \!\sum_{0\leq |j_1|, \dots, |j_k| \leq d} \! \| A_{j_1}\dots A_{j_k} f\|_{p} \| R_{j_1\dots j_k} g\|_{p'} +\sum_{m=1}^{k-1} |c_m|\|\AH^{m/2 } f\|_p \|\AH^{-\frac{k-m}{2}} g\|_{p'}\\
    &\leq C \|f\|_{W_{\AH}^{k,p}(\mathbb R^{d+1})},
\end{align*}
where we used \eqref{ind assump} and the boundedness of Riesz transforms in the last inequality. This completes the proof that $ W_{\AH}^{k, p}(\mathbb R^{d+1})\subset L_{\AH}^{k, p}(\mathbb R^{d+1})$ for $k\ge 1$.
\end{proof}

\begin{prop} Let $p \in (1, \infty) $.
    The Riesz transforms $R_j$ $(-d\le j\le d)$  are bounded on the space $L_{\AH}^{\alpha, p}(\mathbb R^{d+1})$.


\end{prop}
\begin{proof}
 By the definition of $L_{\AH}^{\alpha, p}(\mathbb R^{d+1})$, it suffices to show for  $0\leq |j |\leq d $ that the operators
\begin{align*}
    T_j&= \AH^{-\alpha/2}A_j  \AH^{-1/2} \AH^{\alpha/2}
    \end{align*}
are bounded on $L^p(\mathbb R^{d+1})$.
By Proposition \ref{composition}, Lemma \ref{symbol of fractional powers} and the fact that the symbol of the operators $A_j$ belongs to the symbol class $G^1$,
the symbols of the operator $T_j$  belong to the symbol class $S^0_{1,0}$.
Hence, the operators $T_j$ are bounded on $L^p(\mathbb R^{d+1})$ for $1<p<\infty$ (see \cite{Taylor, Taylor91}), which proves the result.\end{proof}

A direct consequence of Proposition \ref{decay in x} is that any function in $L_{\AH}^{\alpha, p}(\mathbb R^{d+1})$ enjoys some decay in  the $x$ direction.

\begin{cor}\label{cor: decay} If $p\in [1, \infty)$, $\alpha>0$ and $f\in L_{ \AH}^{\alpha, p}(\mathbb R^{d+1})$, then $|x|^\alpha f $      belongs to $L^{p}(\mathbb R^{d+1})$.
\end{cor}

Next, we show the relations between space $L_{\AH}^{\alpha,p}$  and spaces  $W^{\alpha, p}(\mathbb R^{d+1})$, $L^{\alpha, p}_{H}(\mathbb R^{d+1}) $ adapted to the Laplacian and Hermite operators, respectively.
\begin{thm}
 Let $\alpha>0$ and $p\in (1, \infty),$ then
\begin{enumerate}
    \item $L^{\alpha, p}_{H}(\mathbb R^{d+1}) \varsubsetneqq L^{\alpha, p}_{\AH}(\mathbb R^{d+1}) \varsubsetneqq W^{\alpha, p}(\mathbb R^{d+1})$.\label{thm-space-inclusion-2}
    \item  If $ f\in W^{\alpha, p}(\mathbb R^{d+1})$ and has compact support, then $f\in  L_{ \AH}^{\alpha, p}(\mathbb R^{d+1}) $. \label{thm-space-inclusion-3}
\end{enumerate}

\end{thm}

\begin{proof} We firstly show the inclusion in $(1)$.  It suffices to verify that the symbols of
\begin{align*}
    (1-\Delta_{\mathbb R^{d+1}})^{\alpha/2} \AH^{-\alpha/2},\;\; \text{and} \;\;  \AH^{\alpha/2} H^{-\alpha/2}
\end{align*}
belong to the symbol class $S_{1,0}^0$ and so  they define bounded operators on $L^p(\mathbb R^{d+1})$, (see \cite{Taylor, Taylor91}).

For the former: since the symbol of the operator $(1-\Delta_{\mathbb R^{d+1}})^{\alpha/2}$ belongs to class $S^\alpha_{1,0}$ and  the symbol of the operator $\AH^{-\alpha/2} $ belongs to class $ G^{-\alpha} $ due to Lemma \ref{symbol of fractional powers} and class $  S^{-\alpha}_{1,0}$ , we obtain that the symbol of the operator $ (1-\Delta_{\mathbb R^{d+1}})^{\alpha/2} \AH^{-\alpha/2}$ belongs to $S_{1,0}^0$. Therefore,  the operator $(1-\Delta_{\mathbb R^{d+1}})^{\alpha/2} \AH^{-\alpha/2}$ is  bounded on $L^p(\mathbb R^{d+1})$.

%
For the later:  it is known for the symbol $q_\alpha(\rho, x, \tau, \xi)$ of the operator $H^{-\alpha/2}$ that  $q_\alpha \in G^{-\alpha}$  in \cite{Thangavelu}  since
\begin{align*}
      |D_z^\beta D_{\tau}^\gamma D_{\xi}^{\delta}q_{\alpha}(\rho, x, \tau, \xi) |&\lesssim (1+|\tau|+|\xi|+|\rho|+|x|) ^{-\alpha-|\beta|-|\gamma|-|\delta|}\\
      &\lesssim (1+|\tau|+|\xi|+|x|) ^{-\alpha-|\beta|-|\gamma|-|\delta|}.
\end{align*}
As the symbol of the operator $\AH^\alpha $ belong to class $ G^{\alpha}$,  the symbol of the operator $\AH^{\alpha/2} H^{-\alpha/2}$ belongs to $S_{1,0}^0$ by  Proposition \ref{composition}, which implies that the operator $\AH^{\alpha/2} H^{-\alpha/2}$ is  bounded on $L^p(\mathbb R^{d+1})$.

Next, we show the nonequivalence between them.  Define
\begin{align*}
    g_1(\rho, x)&=\frac{1}{
    (1+\rho)^{\frac{1}{p}+\alpha}}\frac{1}{(1+|x|)^{\frac{1}{p}+\alpha}}, & \quad 
     f_1(\rho, x)&=(I-\Delta_{\mathbb R^{d+1}})^{-\alpha/2} g_1(\rho, x),
    \\
 g_2(\rho, x)&=\frac{1}{(1+\rho)^{\frac{1}{p}+\alpha}} \Phi_{\mu}(x), & \quad 
  f_2(\rho, x)&=H^{-\alpha/2} g_2(\rho, x).
\end{align*}
We claim that
\begin{align*}
    f_1\in W^{\alpha, p}(\mathbb R^{d+1})\setminus L^{\alpha, p}_{\AH}(\mathbb R^{d+1}),\; \; \text{and} \;\;
     \ f_2 \in L^{\alpha, p}_{\AH}(\mathbb R^{d+1})\setminus L^{\alpha, p}_{H}(\mathbb R^{d+1}).
\end{align*}
Let $G_\alpha(z)$ be the kernel  of the opertor $(I-\Delta_{\mathbb R^{d+1}})^{-\alpha/2}$, that is,
\begin{align*}
    G_\alpha(z)=\frac{1}{(4\pi)^{\alpha/2} \Gamma(\alpha/2)} \int_0^\infty \ee^{-\frac{\pi|z|^2}{t}-\frac{t}{4\pi}} t^{\alpha/2-(d+1)/2} \dd t.
\end{align*}
On the one hand, it is easy to see that $g_1 \in L^p(\mathbb R^{d+1})$, hence $f_1\in W^{\alpha, p}(\mathbb R^{d+1})$.
On the other hand, since $ G_\alpha(z)$ is  positive,
\begin{align*}
    f_1(\rho, x) = \int_{\mathbb R^{d+1}}  G_\alpha(z' )  g_1(z-z') \dd z'
    \geq (2+|\rho|)^{-1/p-\alpha}(2+|x|)^{-1/p-\alpha}\int_{|z|<1} G_\alpha (z') \dd z', 
\end{align*}
which implies that $|x|^\alpha f_1 \notin L^{p}(\mathbb R^{d+1})$, and therefore  we have $f_1\notin  L^{\alpha, p}_{\AH}(\mathbb R^{d+1})$ by Corollary \ref{cor: decay}.

Similarly, using the fact that $g_2 \in L^p(\mathbb R^{d+1})$ holds, we have  $f_2 \in  L^{\alpha, p}_{\AH}(\mathbb R^{d+1})$. At the same time, we have
\begin{align*}
    f_2(\rho, x) \geq (2+|\rho |)^{-1/p-\alpha} \int_{|z|<1} \Psi_\alpha (z') \dd \rho' \dd x', 
\end{align*}
which implies that  $|\rho|^\alpha f_2 \notin L^{p}(\mathbb R^{d+1})$,  and therefore  we have   $f_2\notin L^{\alpha, p}_{H}(\mathbb R^{d+1})$.

Lastly,  the statement \eqref{thm-space-inclusion-3} follows from \eqref{thm-space-inclusion-2} and  \cite[Theorem 3 (iii)]{BT06}.\qedhere
 
\end{proof}

\section{Some Integral Inequalities adapted to the operator $\AH$}
\label{s:5-fun-ineqs}
In this section, we obtain some integral inequalities associated to the partial harmonic oscillator $\AH$.
\subsection{Hardy--Littlewood--Sobolev inequality. }  Let $0<\alpha<d+1$.  By \eqref{control function}, we have
 \begin{align}\label{HLS1}
     \AH ^{-\alpha/2} f(z)\leq C \int_{\mathbb R^{d+1}} \frac{|f(z')|}{|z-z'|^{d+1-\alpha}} \dd z',
 \end{align}
then  we have

 \begin{prop}\label{Hardy--Littlewood--Sobolev inequality}
    Let $p, q>1$  and $0<\alpha<d+1$ with $\frac{1}{q}=\frac{1}{p}-\frac{\alpha}{d+1}$, then the operator  $\AH^{-\alpha/2}$ is bounded from $L^p(\mathbb R^{d+1})$  to $L^q(\mathbb R^{d+1})$.
\end{prop}

\begin{proof}
This is obvious from the rough estimate \eqref{HLS1} and the classical Hardy--Littlewood--Sobolev inequality in \cite{LIebL}.
	\end{proof}


In fact,  by  \eqref{control function}, the integral kerel of the operator $\AH^{-\alpha/2} f$  is controlled by an integral function which has exponential decay away from the zero,  which implies the following refined estimates.
\begin{thm}\label{generalized Hardy--Littlewood--Sobolev inequality} Let $0<\alpha<d+1$. Then the following holds:
\begin{enumerate}
    \item \label{thm-HLS1} there exists a constant $C>0$ such that
    \begin{align*}
        \| \AH^{-\alpha/2} f\|_{q} \leq C \|f\|_1,
    \end{align*}
    for all $f\in L^1(\mathbb R^{d+1})$ if and only if $1\leq q <\frac{d+1}{d+1-\alpha}$.
    \item \label{thm-HLS2}  there exists a constant $C>0$ such that
    \begin{align*}
        \| \AH^{-\alpha/2} f\|_{\infty} \leq C \|f\|_p
    \end{align*}
    for all $f\in L^p(\mathbb R^{d+1})$ if and only if $p>\frac{d+1}{\alpha}$.
\item \label{thm-HLS3} If $1<p<\infty$, $1<q<\infty$ and $\frac{1}{p}-\frac{\alpha}{d+1} \leq \frac{1}{q}<\frac{1}{p}$, then there exists a constant $C>0$ such that
\begin{align*}
    \| \AH^{-\alpha/2} f\|_{q} \leq C \|f\|_p
\end{align*}
for all $f\in L^{p}(\mathbb R^{d+1})$.

\end{enumerate}

\end{thm}
\begin{proof}For the case \eqref{thm-HLS1}.  By generalized Minkowski's inequality, we obtain for function $f\in L^1(\mathbb R^{d+1})$ that 
\begin{align*}
   \int_{\mathbb R^{d+1}} ( \AH^{-\alpha/2} f) ^q (z) \dd z \leq \left(\int_{\mathbb R^{d+1}}\left(  \int_{\mathbb R^{d+1}} K_{\alpha/2}(z, z')^q \dd z \right)^{1/q}| f(z')|\dd z'\right)^{q}
\end{align*}
By \eqref{kernel-ineq} and \eqref{control function}, we have
\begin{align*}
     \int_{\mathbb R^{d+1}} K_{\alpha/2}(z, z')^q \dd z
    \lesssim \int_{D_{-}} \frac{\dd z}{|z-z'|^{q(d+1-\alpha)}}+\int_{D_{+}} \ee^{-q|z-z'|^2/{16}} \dd z,
\end{align*}
where $ D_{-}=\{|z-z'|<1\}$ and $D_{+}=\{|z-z'|\geq 1\}$.
The right hand side in the above estimate  is finite if $1\leq q <\frac{d+1}{d+1-\alpha}$.
For the converse, note that as $|z-z'|<1$, we have
\begin{align}
    K_{\alpha/2}(z, z') \geq C_\alpha \ee^{-|x+x'|^2} \int_0^1 t^{\alpha-1-\frac{d+1}{2}} \ee^{- |z-z'|^2/t}\dd t  \geq C_\alpha \frac{\ee^{-|x+x'|^2}}{|z-z'|^{d+1-\alpha}}.
    \label{lower bound for K}
\end{align}
Let $f_n$ be an approximation of identity. Then we have
\begin{align*}
  \int_{\mathbb R^{d+1}} \left(\int_{\mathbb R^{d+1}} \frac{\ee^{-|x+x'|^2}}{|z-z'|^{d-\alpha}} f_n(z')  \dd z'\right)^q   \dd z
 \xrightarrow[n\to\infty]{} \int_{|z| \leq 1} \frac{\ee^{-q|x|^2}}{|z|^{q(d+1-\alpha)}} \dd z,
\end{align*}
which is $\infty$ for $q(d+1-\alpha) \geq d$, and completes the proof of the necessity that $1\leq q\leq d/(d+1-\alpha)$.

For the case $(2)$. By H\"older's inequality, we get
\begin{align*}
  \left| \int_{\mathbb R^{d+1}} K_{\alpha/2 }(z, z') f(z') \dd z'\right|   \leq \|f\|_{p} \left(\int_{\mathbb R^{d+1}} K_{\alpha/2 }(z, z')^{p'} \dd z'\right)^{1/p'}.
\end{align*}
Using a similar argument as in the proof of the case $(1)$, we know that the right hand side is finite when $p>\frac{d+1}{\alpha}$.

Conversely, by choosing 
\begin{align*}
    f(z)=\begin{cases}
        |z|^{-\alpha} (\log \frac{1}{|z|})^{-\frac{\alpha}{d+1}(1+\varepsilon )}, & \text{if } |z|\leq 1/2, \\
        0, & \text{if } |z|\geq  1/2, 
    \end{cases}
\end{align*}
then we have $f\in L^p(\mathbb R^{d+1})$ for all $p\leq (d+1)/\alpha$. However, the function $\AH^{-\alpha/2}f$ is essentially unbounded  since we have  by \eqref{lower bound for K} that
\begin{align*}
   \AH^{-\alpha/2} f(0)  \geq C\int_{|z'|\leq 1/2} |z|^{-(d+1)}\left (\log \frac{1}{|z'|}\right)^{-\frac{\alpha}{d+1}(1+\varepsilon )} \dd z' =\infty,
\end{align*}
where $ \varepsilon $ is small.

 The case \eqref{thm-HLS3} now follows from  Proposition \ref{Hardy--Littlewood--Sobolev inequality}, the inequality \eqref{thm-HLS1}, the inequality \eqref{thm-HLS2} and the Riesz--Thorin interpolation theorem.
\end{proof}
\begin{rem}\label{remark Hardy--Littlewood--Sobolev inequality} By Corollary \ref{kernel of H+-2} and the similar argument as in the above proof,  we have  the following result: 
Let $p, q>1$ and $0<\alpha<d+1$ with $\frac{1}{p}-\frac{\alpha}{d+1} \leq \frac{1}{q}<\frac{1}{p}$, then there exists a constant $C$ such that  for all $f\in L^{p}(\mathbb R^{d+1})$, we have
\begin{align*}
    \| (\AH+2)^{-\alpha/2} f\|_{q} &\leq C \|f\|_p, \\
     \| (\AH-2)^{-\alpha/2} f\|_{q} &\leq C \|f\|_p, \ d\geq 3.
\end{align*}

\end{rem}

\subsection{Gagliardo--Nirenberg--Sobolev inequality}

We define the  gradient operator associated to the operator $\AH$ as follows:
\begin{align*}
    \nabla_{\AH} f:=(A_0f, A_1 f, \dots, A_d f,  A_{-1} f, \dots, A_{-d}  f).
\end{align*}
\begin{thm}
Let $d\geq 3$ and $1<p, q <\infty$ satisfy $\frac{1}{p}-\frac{1}{d+1} \leq \frac{1}{q}<\frac{1}{p}$. Then for any $f\in L_{\AH}^{1,p}(\mathbb R^{d+1})$, we have
\begin{align*}
    \|f\|_q \leq C \|\nabla_{\AH} f\|_p. 
\end{align*}
\end{thm}

\begin{proof} By duality and \eqref{duality identity}, we have
\begin{align*}
    \|f\|_{q} \leq \sum_{-d\le j\le d} \|R_j f\|_{q}.
\end{align*}
By Lemma \ref{commute lemma}, we obtain $R_{0} f =\AH^{-1/2} A_{0} f$ and   
\begin{align*}
    R_j f=(\AH+2 \operatorname{sgn}j)^{-1/2} A_j f,  \
   \qquad  1\le |j|\le d .
\end{align*}
Hence, by Remark \ref{remark Hardy--Littlewood--Sobolev inequality} and  Theorem \ref{generalized Hardy--Littlewood--Sobolev inequality},  we obtain for  $\frac{1}{p}-\frac{1}{d+1} \leq \frac{1}{q}<\frac{1}{p}$ that 
\begin{align*}
     \|f\|_{q} &\leq \sum_{-d\le j\le d} \|R_{j} f\|_{q}\\
     &\leq \sum_{1\leq |j| \leq d} \|(\AH+2\operatorname{sgn}j)^{-1/2} A_j f\|_{q}+ \| \AH^{-1/2} A_{0} f\|_{q}\\
     &\leq \sum_{-d\le j\le d} \|A_jf\|_{p} = \|\nabla_{\AH} f\|_p. 
\end{align*}
 This completes the proof.
\end{proof}

\begin{rem} The classical Gagliardo--Nirenberg--Sobolev inequality in $\mathbb R^{d+1}$ holds with $\frac1q = \frac1p - \frac \alpha{d+1}$. The result here holds for a larger range of $(p, q)$ due to the extra decay property of functions $f\in L_{\AH}^{1,p}(\mathbb R^{d+1})$.

\end{rem}
\subsection{Hardy's inequality }Recall that the Hardy's inequality is:
\begin{align}\label{Hardy's inequality for laplacian}
    \||z|^{-\alpha} f(z)\|_{L^p(\mathbb R^{d+1})} \lesssim \|(-\Delta_{\mathbb R^{d+1}})^{\alpha/2} f\|_{L^p(\mathbb R^{d+1})}, \;\; 0<\alpha<\frac{d+1}{p},
\end{align}
see  \cite{Triebel}.
We extend classical Hardy's inequality for the Laplacian operator $-\Delta_{\mathbb R^{d+1}}$ to the operator $\AH$ as follows:

\begin{thm} Let $1<p<\infty$. Then, for any $f\in C_0^\infty (\mathbb R^{d+1})$, we have
\begin{align*}
    \||z|^{-\alpha} f(z)\|_{L^p(\mathbb R^{d+1})} \lesssim \|\AH^{\alpha/2} f\|_{L^p(\mathbb R^{d+1})},\  0<\alpha<\frac{d+1}{p} .
\end{align*}
In particular, the inequality 
\begin{align*}
    \||z|^{-1} f(z)\|_{L^p(\mathbb R^{d+1})} \lesssim \|\nabla_{\AH} f\|_{L^p(\mathbb R^{d+1})}
\end{align*}
holds when $1<p<d+1$.
  \end{thm}

 \begin{proof}
 Note that the symbols of $(1-\Delta_{\mathbb R^{d+1}})^{\alpha/2}$ and $\AH^{-\alpha/2}$ belong to $S_{1,0}^\alpha$ and $S^{-\alpha}_{1,0}$ respectively. The composition law gives that the symbol of $(1-\Delta_{\mathbb R^{d+1}})^{\alpha/2}\AH^{-\alpha/2} $ belongs to $S^{0}_{1,0}$,  which implies that $(1-\Delta_{\mathbb R^{d+1}})^{\alpha/2}\AH^{-\alpha/2 }$  are bounded in $L^p(\mathbb R^{d+1})$ for $1<p<\infty$, (see \cite{Taylor, Taylor91}). In addition,  $(-\Delta_{\mathbb R^{d+1}})^{\alpha/2}(1-\Delta_{\mathbb R^{d+1}})^{-\alpha/2}$ are bounded in $L^p(\mathbb R^{d+1})$ for $1<p<\infty$ from \cite{Stein}. Hence, by the inequality \eqref{Hardy's inequality for laplacian} we have 
 \begin{align*}
     \||z|^{-\alpha} f(z)\|_{L^p(\mathbb R^{d+1})} 
     &\lesssim \| (-\Delta_{\mathbb R^{d+1}})^{\alpha/2} f\|_{L^p(\mathbb R^{d+1})}\\
     &\lesssim \| (-\Delta_{\mathbb R^{d+1}})^{\alpha/2}(1-\Delta_{\mathbb R^{d+1}})^{-\alpha/2}(1-\Delta_{\mathbb R^{d+1}})^{\alpha/2}\AH^{-\alpha/2 }\AH^{\alpha/2 }f\|_{L^p(\mathbb R^{d+1})}\\
     &\lesssim \|\AH^{\alpha/2 }f\|_{L^p(\mathbb R^{d+1})}.
      \end{align*}
For $\alpha=1$, using the first inequality in \eqref{inverse Riesz transform}, we have 
\begin{align*}
    \||z|^{-1} f(z)\|_{L^p(\mathbb R^{d+1})}\lesssim \|\AH^{1/2} f\|_{L^p(\mathbb R^{d+1})} \lesssim \sum_{0<|j|\leq d} \| A_jf \|_{L^p(\mathbb R^{d+1})}
    \lesssim \|\nabla_{\AH} f\|_{L^p(\mathbb R^{d+1})}.
\end{align*}
Hence, we completes the proof.
\end{proof}

\newcommand{\doi}[1]{DOI: \href{https://doi.org/#1}{\texttt{#1}}}
\newcommand{\zbl}[1]{Zbl: \href{https://zbmath.org/?q=an\%3A#1}{\texttt{#1}}}
\newcommand{\arxiv}[1]{arXiv: \href{https://arxiv.org/abs/#1}{\texttt{#1}}}

\end{document}